%% file: ex_article.tex
\documentclass{siamart250211}

\input{ex_shared}

\ifpdf
\hypersetup{
  pdftitle={Temperature induced tipping in a two-box ocean circulation model},
  pdfauthor={J. Noory}
}
\fi

\usepackage{amsfonts, amsmath,amssymb, comment}
\usepackage{subcaption}

\begin{document}

\maketitle

\begin{abstract}
 The Atlantic Meridional Overturning Circulation (AMOC) is often analyzed using low-order box models to understand tipping points. Historically, these studies focus on freshwater flux as the primary bifurcation parameter, treating the temperature gradient as a fixed restoring target. However, the erosion of the equator-to-pole temperature contrast due to polar amplification suggests that thermal forcing should be treated as a dynamic control parameter. In this study, we use Cessi’s reduced box model to map the global bifurcation structure of the thermohaline circulation. We relax the assumption of a fixed thermal background and analyze the system’s behavior under joint thermal and haline forcing. We prove the existence of a cusp bifurcation, identifying the specific geometry of pitchfork and saddle-node bifurcations that bound the stable regime. This geometric characterization reveals that thermal erosion acts as a distinct mechanism for destabilization, capable of driving the system across critical thresholds even in the absence of anomalous freshwater forcing.
\end{abstract}




\section{Introduction}
    The Atlantic Meridional Overturning Circulation (AMOC) is a part of the global ocean current, a mechanism by which the ocean circulates heat and regulates regional climate patterns. Many processes influence the AMOC such as mixing and wind currents. An idealized model of the AMOC functions by sending warm, salty water from the tropics northward toward the Arctic. As this warmer water reaches colder regions, it cools and becomes more saline, both contributing to the water becoming denser, making it sink into the deep ocean. The transformed water then flows southward at depth, eventually warming as it moves toward the equator. As it warms again, its density decreases, and it rises back to the surface, completing a large-scale thermohaline circulation loop. This process is slow, taking roughly 1,000 years for a parcel of water to complete the full cycle \cite{NOAA}. 

    It's posited that freshwater injections into the North Atlantic ocean, disrupting the AMOC pattern, are a driver of abrupt climate transitions. This is evidenced in the paleoclimate record, seen in ice core data collected from Greenland, which reveals fluctuations between warming and cooling ocean temperatures in the northern hemisphere occurring in the last glacial period. The climatic fluctuations are referred to as Dansgaard-Oescher (D-O) events, characterized by rapid warming followed by gradual cooling. These events  suggest that Earth's climate oscillated between  quasi-stable climate regimes.

    Researchers use box models to study the AMOC. In the 1960s, oceanographer Henri Stommel introduced a box model to demonstrate how thermohaline circulation arises from density differences between the northern and southern oceans. Influenced by external freshwater forcing and governed by relaxation processes, the model shows how density differences, defined by competing forcing of temperature and salinity, drive the speed of the circulation. Often used in study in which the circulation is subjected primarily to freshwater forcing, this model exhibits bistability for a particular range of parameters. The bifurcation diagram exhibits two stable branches, one dominated by a temperature mode and the other dominated by salinity.  Moreover, its bifurcation diagram demonstrates how freshwater forcing affects the overturning circulation, with an increase in freshwater forcing tending towards a globally stable regime. 

    Two-box models are often used as a basis for methods experimentation, particularly in paleoclimate studies. The Younger Dryas (YD) transition event is the last D-O event of the last glacial period. Motivated to back the hypothesis that the YD transition was triggered by an increase in freshwater flux by melting ice dams, causing a deluge from ancient Lake Agassiz, oceanographer Paola Cessi \cite{Cessi1994} used Stommel's model. She contributed a reformulation of salinity dynamics and assumed a large enough timescale difference between temperature and salininty as a treatment to reduce Stommel's model to one dimension. The fast-slow system induced by the reformulation is dominated by the salinity dynamics, thus the temperature gradient is often assumed to be fixed at the relaxation target of the equilibrium the temperature dynamics in the relaxation process. She derived critical thresholds for amplitude and duration of freshwater influx, demonstrating that a sufficient amplitude of influx could cause a dynamic bifurcation of the system for long enough such that a flow, starting in the basin of attraction of one stable state, could overcome the unstable equilibrium and be influenced by the dynamics into the alternate stable state's basin.  

    Although salinity is traditionally identified as the dominant variable governing hysteresis in the thermohaline circulation, the stabilizing role of temperature warrants critical re-examination. Historically, the paradigm of AMOC tipping has centered on freshwater forcing as the control parameter \cite{Stommel1961, Cessi1994}. However, recent refinements suggest the system possesses greater salinity-driven stability than initially assumed (Lambert et al. 2016). Conversely, thermal forcing plays a dominant role in determining inflow sensitivity and overall circulation strength (Schloesser et al. 2020). Therefore, there is a pressing need to re-evaluate the system's dependence on the meridional temperature gradient. As polar amplification continues to erode the equator-to-pole thermal contrast, treating the temperature gradient as a primary bifurcation parameter allows us to capture a tipping mechanism that is more directly aligned with observed climate trends than traditional freshwater hosing experiments.

The motivation to elevate the temperature gradient to a primary bifurcation parameter is underscored by both paleoclimate analogues and modern observation. During intervals of "equable climate" such as the early Pliocene, the meridional temperature gradient was significantly weaker than today, yet the overturning circulation persisted, potentially in a structurally distinct stability regime \cite{Fedorov}. This weakened gradient is analogous to the modern phenomenon of polar amplification, where Arctic warming outpaces the global average, effectively eroding the equator-to-pole density contrast that drives the thermal component of the AMOC \cite{Serreze}. As this thermal driving force diminishes, the system may approach a "thermal threshold" distinct from the traditional freshwater-induced saddle-node bifurcation. Consequently, investigating the global bifurcation structure with respect to the thermal gradient allows us to map the stability landscape relevant to a warming—rather than merely freshening—world.

 In the context of low-order modeling, \cite{Rahmstorf1996} examined the impact of these temperature feedbacks using an idealized model with Newtonian restoring. In this framework, sea surface temperature is subject to linear restoring toward a prescribed atmospheric state, while salinity evolves prognostically under surface freshwater forcing. By varying the freshwater flux, \cite{Rahmstorf1996} mapped the equilibrium structure and stability of the circulation. Crucially, this establishes a separation of timescales: temperature acts as a fast, linearly damped variable providing negative (stabilizing) feedback on circulation anomalies, while salinity supplies the slow, nonlinear positive feedback responsible for multiple equilibria and abrupt transitions.

    This paper contributes a geometric characterization of the role of temperature in shaping the stability and tipping behavior of the thermohaline circulation within Cessi’s box model framework. While salinity is well established as the dominant driver of multiple equilibria and abrupt transitions, temperature is typically treated either as a fast, slaved variable or as a fixed boundary condition. Here, we relax this assumption and examine how variations in temperature boundary conditions modify the bifurcation structure of the system. Using a cusp bifurcation framework, we jointly vary temperature and salinity forcing to identify the geometry of the pitchfork bifurcations that bound the bistable regime. We show that temperature exerts a stabilizing influence by reshaping the folds of the bifurcation surface, thereby altering the range of salinity perturbations capable of inducing tipping. This perspective reframes temperature not merely as a passive, stabilizing feedback, but as a control parameter that modulates the system’s sensitivity to freshwater forcing.

\vspace{1cm}
\section{Background}
    We start with an explanation of Stommel's model of thermohaline circulation as presented in Paola Cessi's 1994 paper \cite{Cessi1994}. The reduced model will be referred to as Cessi's model in the proceeding sections.
    
    \subsection*{The Model}
    We first describe the governing equations and exchange law, then introduce a nondimensional form that facilitates bifurcation and stability analysis. Assume there are two well-mixed liquid reservoirs with the following external forcing: a freshwater influx into the compartments as well as an ambient temperature relaxation target, which may differ for each box. Assume that density is described by a linear combination of temperature and salinity with characteristic contraction and expansion rates. Finally, assume that both reservoirs interact through an exchange that is driven by the density difference between boxes. This is Stommel's box model \eqref{Stommel Model}, and it captures thermohaline circulation. 
    
    In what follows, we adopt the nonlinear exchange formulation introduced by Cessi, which smooths Stommel’s original flow law while preserving its essential bistable structure. The model is described as follows
    
     \begin{align}\label{Stommel Model}
        \frac{d\Delta T}{dt} &= -t_r^{-1}(\Delta T - \theta ) - Q(\Delta\rho) \Delta T  \notag \\ 
        \frac{d\Delta S}{dt} &= \frac{F(t)}{H}S_0 - Q(\Delta\rho) \Delta S  \\ 
        \Delta \rho &= \alpha_S \Delta S - \alpha_T \Delta T. \notag
    \end{align}

    The model expresses thermohaline circulation between two compartments representing a two-box ocean for northern and southern latitudes with a system of ordinary differential equations. The system tracks the time evolution of temperature differences $\Delta T$ and salinity difference $\Delta S$ between the boxes. Temperature and salinity differences are forced by different processes. The temperature forcing process in this model is Newton's law of cooling, represented by a relaxation of $\Delta T$ towards the temperature difference equilibrium, $\theta$. The rate at which the boxes equilibrate to $\theta$ is $t_r^{-1}$, the characteristic relaxation timescale. The salinity difference is forced with a prescribed freshwater flux $F(t)$, representing the imbalances between evaporation and precipitation plus freshwater runoff from rivers and glaciers into the North Atlantic. $F(t)$ acts indiscriminately on the whole box of depth $H$ with reference salinity state $S_0$.  
    
    The interaction between the boxes is described by exchange function $Q$, which is driven by the density difference between the two, $\Delta \rho$. The density difference, $\Delta\rho$, depends on a linear combination of temperature and salinity differences through thermal expansion coefficient $\alpha_T$ and haline contraction coefficient $\alpha_S$. In Stommel's model, the exchange is described as an absolute value term, representing the magnitude of the volume transport (or flow rate) between the equatorial and polar boxes, dictating the speed at which heat and salt are mixed regardless of the circulation's direction. In contrast, Cessi's model used a squared flow rate parameter to describe the strength of the exchange, justified by a quadratic scaling of volume transport as an effect of the density gradient scaling with the current velocity and width of the basin. This formulation of the flow rate means that the model is smooth and differentiable everywhere, allowing her to analytically derive the potential landscape in which the ocean state resides.
    
    The exchange function $Q$ is of the following form, where $M$ is the volume of the box, $t_d$ is the characteristic timescale of salinity diffusion, and $q$ represents a proportionality constant of the density difference: 
    \begin{align}
        Q(\Delta \rho) = t_d^{-1} + M^{-1} q(\Delta \rho)^2. \label{eqn: exchange function}
    \end{align}

    The following change of variables is used to nondimensionalize the system. The new variable $x$ tracks a dimensionless temperature difference, and the variable $y$ tracks the dimensionless salinity difference. Time $t$ represents a scaled time expanded to the characteristic diffusion timescale. The new variables are
    \[ x = \frac{\Delta T}{\theta}, \ \ \ y = \frac{\alpha_S \Delta S}{\alpha_T \theta}, \ \ \ t = t_d t'.\]
    
    \noindent Then equations \eqref{Stommel Model} become
    \begin{align}
        x' &= -\alpha(x-1) - x(1+\mu^2(x-y)^2) \label{nondimensional Cessi model}\\
        y' &= p(t) -y(1+\mu^2(x-y)^2). \notag
    \end{align}
    
   \noindent Here $\alpha=\frac{t_d}{t_r}$, describes the ratio of characteristic timescales of diffusion to relaxation. The parameter $\mu^2 = qt_d(\alpha_T \theta)^2/M$ describes the ratio of diffusive timescale to advective timescale. In this parameter expression, the transport of total volume $M$ represents the typical width of the Western boundary current rather than the entire width of the North Atlantic so to keep consistent with the assumption that present day transport is dominated by the advective process. Lastly, $p$ is the nondimensional freshwater flux with $p(t) = \frac{\alpha_S S_0 t_d}{\alpha_T \theta H} F(t)$.  The freshwater flux is approximated by a constant average yearly freshwater flux with an added pulse of influx. With an average estimate for the North Atlantic fresh water flux $\bar F = 2.3$ m yr$^{-1}$ and $\theta = 20^{\circ}$C, then $p\approx 1$. 

   Using Cessi’s parameter estimates of characteristic timescales, $t_d \approx 219$ years and $t_r \approx 25$ days, the ratio $\alpha \sim \mathcal{O}(10^3)$, indicating a strong separation of timescales between temperature and salinity evolution. This motivates Cessi's contribution to introduce a reduced one-dimensional system, where the salinity timescales are assumed, thus temperature is assumed to relax instantaneously to its critical manifold $x=1$. As a starting point, we will consider the freshwater forcing $p$ to be constant. The one-dimensional model is 
   
   \begin{equation}\label{eqn: Cessi orig model}
       y' = p - y(1+\mu^2(1-y)^2). 
   \end{equation}

   This one-dimensional system admits a potential function. The salinity dynamics on the critical manifold are governed by the equation $\frac{dy}{dt'} = -V'(y)$, where
    \begin{equation}\label{eqn: potential function}
        V(y) = \frac{y^2}{2} + \mu^2\left(\frac{y^2}{2}-\frac{2}{3}y^3 +\frac{y^4}{4}\right)+ yp.
    \end{equation}
    The stable equilibria of the salinity dynamics correspond to the local minima of the Lyapunov potential function $V(y)$. Formally, linear stability is determined by the curvature of the potential, requiring $V''(y) > 0$. Physically, this condition implies that the restoring force of the thermal feedback dominates the destabilizing advection of salinity. The depth of the wells in the potential function corresponds to the resilience of each equilibrium. Figure \ref{fig: 1D_Landscape_Pvary} demonstrates the potential energy landscape of the salinity dynamics for two freshwater forcing values: an average freshwater forcing selected by Cessi (left) and a tipping value identified by Cessi (right). The left panel exhibits two competing stable configurations, with transitions, as $p$ varies, only occurring when the barrier (the repelling state) is overcome. This means that with enough freshwater forcing, a saddle-node bifurcation occurs, creating the opportunity for the flow to enter the basin of attraction of the alternate stable state. This saddle-node bifurcation is well studied. 

    \begin{figure}[H]
        \centering
        \includegraphics[width=0.9\linewidth]{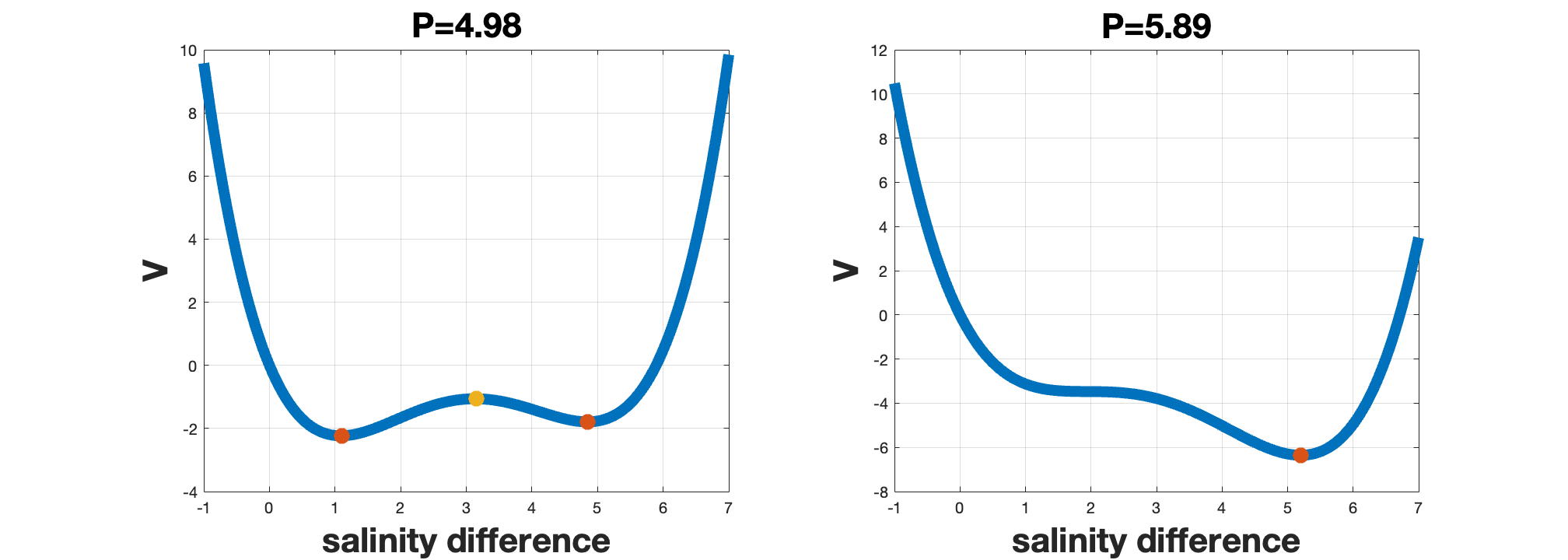}
        \caption{Potential energy landscapes for bistable ($P=4.98$) and monostable ($P=5.89$) freshwater forcing regimes, illustrating the potential barrier separating the competing stable states.}
        \label{fig: 1D_Landscape_Pvary}
    \end{figure}

    Stable configurations of circulation

    \begin{figure}[H]
        \centering
        \includegraphics[width=0.8\linewidth]{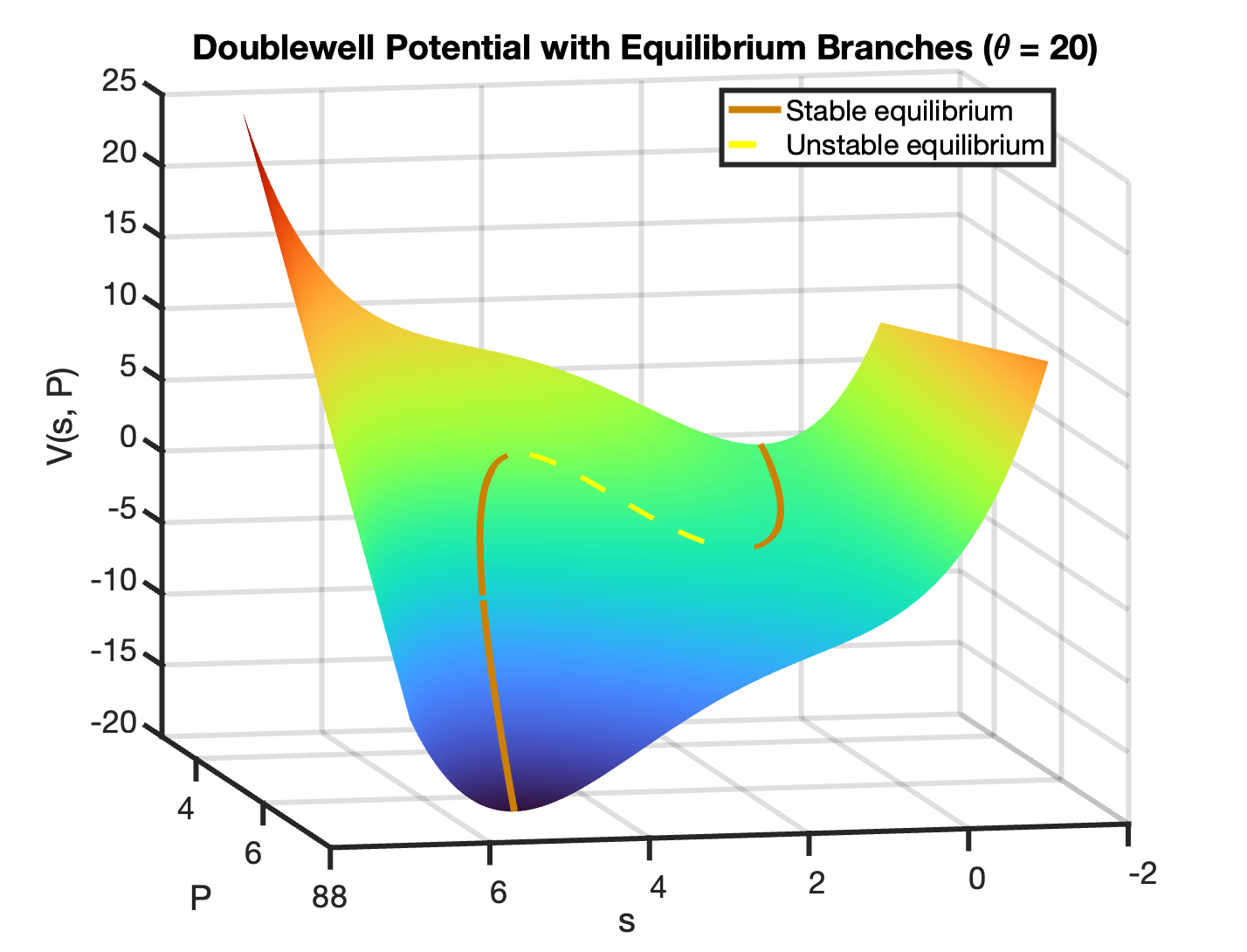}
        \caption{Bifurcation diagram showing the stable (solid) and unstable (dashed) equilibrium branches of the circulation as a function of freshwater forcing $P$. The upper and lower stable branches correspond to salinity-dominated and thermally-dominated modes, respectively.}
        \label{fig:LandscapeSurface_varyP}
    \end{figure}

    Examining the potential landscape for an extended range of freshwater forcing $P$ reveals that a bistable regime is supported only for a finite range of $P$ values, as depicted in 
    Figure \ref{fig:LandscapeSurface_varyP}. The stable branches of the bifurcation diagram represent the two circulation modes, one at lower $P$ values corresponding to a temperature-dominated circulation, and the other at high $P$ values dominated by salinity. The latter of these is associated with reduced effective exchange between the boxes.

\section{Two-parameter reformulation of Cessi's Model}
We now reformulate Cessi's model to incorporate target temperature difference equilibrium as a parameter. The reformulation permits an analysis of the dynamics of thermohaline circulation subject to changing temperatures in northern and southern oceans. Using this reformulation, we analyze the role of changing ocean temperatures on the dynamics of thermohaline circulation.
To state our reformulation, we pull back to the unnormalized temperature difference and salinity difference using the the relations \eqref{eqn: cessi change of variables} and follow Cessi's timescale-reduction technique.

Redimensionalizing \eqref{nondimensional Cessi model} using the relations 
\begin{align}
    \theta\frac{dx}{dt'} = \frac{d \Delta T}{dt'}  \hspace{1cm} \frac{\alpha_T \theta}{\alpha_S}\frac{dy}{dt'} = \frac{d\Delta S}{dt'} \label{eqn: cessi change of variables}
\end{align}
leads to the system
\begin{align}
     \frac{d\Delta T}{dt'} &= -\alpha(\Delta T - \theta) - \Delta T(1+\beta(\Delta T- \lambda \Delta S)^2 ) \label{eqn:temperature difference}\\
     \frac{d\Delta S}{dt'} &=P(t) - \Delta S(1+\beta(\Delta T -\lambda \Delta S)^2 )\label{eqn:salinity difference}.
\end{align}
Here, $\lambda = \frac{\alpha_S}{\alpha_T}$, $\beta = \frac{qt_d \alpha_T^2}{V}$, and $P(t) = \frac{S_0t_d}{H}F(t)$. Notably, we've recovered the temperature difference equilibrium $\theta$, which represents the difference of ambient temperatures surrounding the northern and southern boxes to which temperatures relax. The time variable $t'$ denotes physical time scaled by the diffusive timescale. The derivative $\frac{d \Delta T}{dt'}$ describes temperature difference change in celcius and $\frac{d \Delta S}{dt'} $ describes salinity difference change in practical salinity unit (psu).

Our reformulation preserves the fast-slow properties of the dynamics. The system may still be reduced to a one-dimensional equation since $\alpha$ is left unchanged from \eqref{nondimensional Cessi model}. Because the temperature difference $\Delta T$ evolves at a faster timescale relative to salinity difference $\Delta S$, we can further simplify the analysis of the system by projecting the salinity dynamics onto the critical manifold of the temperature difference, where the fast variable is at quasi-equilibrium $\Delta T = \theta$. This leads to the following lemma. 

\begin{lemma}[Model Reformulation] \label{lemma: new stommel formulation}
    For a large enough timescale distinction, a one-dimensional model of thermohaline circulation, \eqref{eqn: Cessi orig model}, may be restated as
    \begin{equation}
        \frac{d\Delta S}{dt'} =P - \Delta S(1+\beta(\theta -\lambda \Delta S)^2 ) \label{eqn: theorem unnormalized salinity dyanmics}
    \end{equation}
    with the parameters $\lambda = \frac{\alpha_S}{\alpha_T}$, $\beta = \frac{qt_d \alpha_T^2}{V}$, and $P = \frac{S_0t_d}{H}\bar F$, where $\bar F$ is an annual average freshwater flux, defined by Cessi.
\end{lemma}

In this context, the parameter $\lambda$ describes the ratio of haline contraction to thermal expansion. The density difference term $(\Delta T - \lambda\Delta S)^2$, now in temperature units, still drives the dimensionless exchange $(1 + \beta(\Delta T - \lambda\Delta S)^2 )$. In effect, we see this exchange is drive by the magnitude of distance of the salinity difference, scaled by $\lambda$, from reference $\Delta T$.  Thus we have a one-dimensional model of the salinity dynamics where the temperature difference equilibrium $\theta$ is a parameter!

A standard bifurcation analysis of the reduced salinity dynamics in \eqref{eqn: theorem unnormalized salinity dyanmics} reveals two things. First, we see two equilibria which retain the the same mode characterizations as Stommel's model, namely that there is a fast thermally-driven mode and a slow, haline-driven mode. When the difference is high,  that variations in the equilibrium temperature difference $\theta$ induce a saddlenode bifurcation analogous to that obtained by varying the freshwater forcing parameter $P$. Thus, in addition to freshwater forcing, the meridional temperature gradient itself acts as a control parameter capable of reorganizing the stability structure of the thermohaline circulation.

\begin{theorem}[Temperature Induced Hysteresis]\label{thm: temperaturesaddle}
Variations of the equilibrium temperature difference $\theta$ in Cessi’s model reforumlation, given in \cref{lemma: new stommel formulation}, induce hysteresis.
\end{theorem}

The structure of this bifurcation is illustrated in Figures \ref{fig: LandscapeSurface_varyTheta}, shows the potential landscape of the salinity difference dynamics as a function of both the state variable and the equilibrium temperature difference $\theta$, for fixed freshwater forcing $P=4.89$. Drawn on the surface is the bifurcation diagram, subject to variations in $\theta$, of the salinity difference $s$, which is defined below in \cref{sect: main result} as a shifted representation of $\Delta S$. This diagram reveals two saddle nodes meaning there exists a hysteresis loop.  

The one–dimensional slices shown in Figure \ref{fig: 1D_Landscape_thetavary} demonstrate how the potential landscape is modified as $\theta$ varies. For a finite range of temperature differences, the landscape admits two local minima separated by an unstable saddle, corresponding to bistable circulation states. Outside of this range, the landscape collapses to a single global minimum, and the circulation becomes globally stable. In particular, the potential landscape corresponding to the value used in Cessi’s analysis, $\theta = 20$, lies within the bistable regime, whereas for smaller values such as $\theta = 17$ the system admits only a single stable equilibrium.

Quantitatively, we find that bistability occurs for equilibrium temperature differences in the range
\[18.6^{\circ} C < \theta < 22.8 ^{\circ} C.\]

\noindent Outside of this interval, the salinity difference dynamics settle into a globally stable regime. As in the case of freshwater forcing, sufficiently sustained variations in $\theta$ can drive the system across basin boundaries in the potential landscape, leading to transitions between circulation states. In the bistable regime, the dynamics preferentially evolve toward the higher–salinity–difference attractor under increased temperature gradients.

Taken together, \ref{fig: LandscapeSurface_varyTheta} and Figures \ref{fig: 1D_Landscape_thetavary} demonstrate that the bistable dynamics of the thermohaline circulation are supported only over a finite range of meridional temperature differences. This result highlights that not only freshwater forcing, but also the equilibrium temperature contrast between the northern and southern ocean, plays a critical role in determining the stability regime of the circulation and its susceptibility to tipping.

The natural question is, then, how varying both $P$ and $\theta$ together modulates the potential landscape and what this means for the dynamics. To address this, in the following section we use Ren\'e Thome's framework for cusp catastrophes \cite{ThomRené1989}.

\begin{figure}[h]
    \centering
    \includegraphics[width=0.8\linewidth]{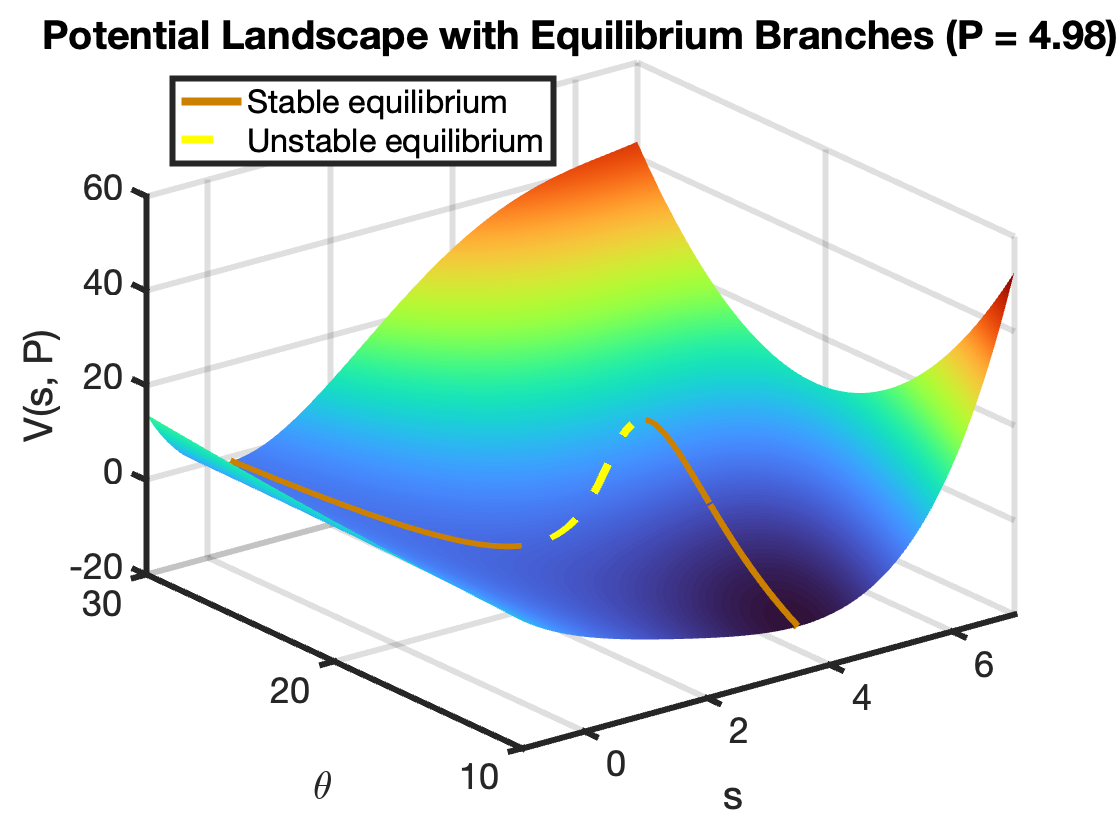}
    \caption{3D potential energy landscape $V(y, \theta)$, visualizing how the depth and number of stability wells evolve as the equilibrium temperature gradient $\theta$ varies.}
    \label{fig: LandscapeSurface_varyTheta}
\end{figure}

\begin{figure}[h]
    \centering
    \includegraphics[width=0.9\linewidth]{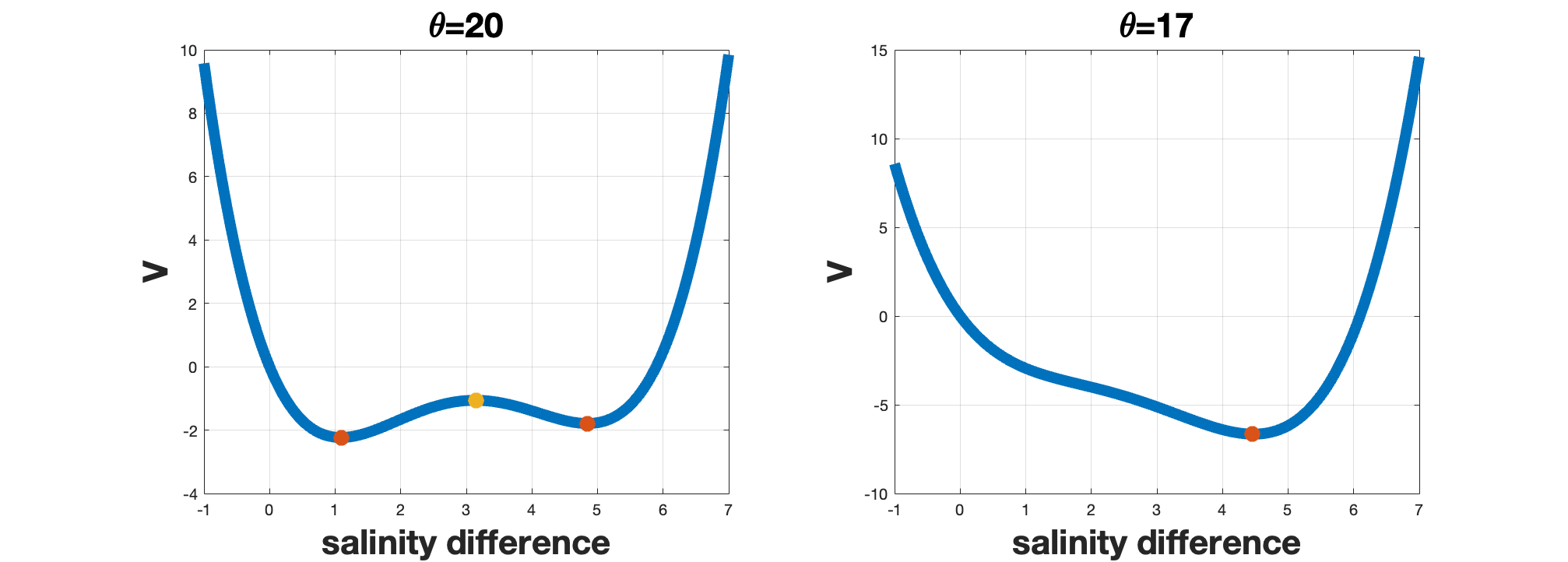}
    \caption{Cross-sections of the potential landscape for $\theta=17$ (globally stable) and $\theta=20$ (bistable), demonstrating how a sufficiently low temperature gradient eliminates the second stable equilibrium.}
    \label{fig: 1D_Landscape_thetavary}
\end{figure}


\section{Cusp analysis of Cessi's Model}\label{sect: main result}

In this section, we show the existence of a cusp bifurcation in Cessi's model under joint variation of $P$ and $\theta$. This is stated in \cref{thm: cusp theorem}, the proof of which can be found in the supplementary material.

René Thome's cusp bifurcation occurs in a cubic ordinary differential equation with two parameters. It has the form
    \begin{align}
        \dot w = h + w (r-w^2) = h + rw - w^3 ,\label{eqn: cusp canonical form}
    \end{align}
    where $w$ is the response variable, and $h$ and $r$ are the parameters to be varied. In the simplest case, both parameters are scalars and the cusp analysis is done graphically by intersecting $g_1 = -h$ and $g_2 = rw-w^3$ and placing $-h$ such that the fixed point is a repeated root of the intersection. In essence, this analysis locates parameter pairs $(r, h)$ where cubic discriminant transitions from negative to positive. To begin, we transform the 1-d model \eqref{eqn: theorem unnormalized salinity dyanmics} into the canonical form used in \cite{ Strogatz, ThomRené1989, Zeeman1976} by using a Tschirnhaus transformation. This transformation is used to depress the quadratic term of a cubic polynomial.

    \begin{lemma}[Tschirnhaus Transformation]\label{lemma: tschirnhaus transformation}
        The Tschirnhaus transformation of equation \cref{eqn: theorem unnormalized salinity dyanmics} is \begin{equation*}
            s=\Delta S+\frac{2\theta}{3\lambda}.
        \end{equation*}
    \end{lemma}
    \noindent We now restate equation \eqref{eqn: theorem unnormalized salinity dyanmics} using the shifted variable $s$.
    \begin{corollary}[Depressed Dynamics]\label{cor: shifted salinity dynamics}
        The salinity difference dynamics \eqref{eqn: theorem unnormalized salinity dyanmics} expressed in the shifted variable $s$ is 
        \begin{equation}\label{eqn: depressed cubic salinity dynamics}
            \dot s = h(\theta, P) + r(\theta) s  - \beta \lambda^2s^3 ,
        \end{equation}
        where $r(\theta) = (\frac{\beta\theta^2}{3}-1)  $ and $h(\theta, P) = P - \frac{2\theta}{3 \lambda}-\frac{2 \beta \theta^3}{27 \lambda}$.
    \end{corollary}

The dynamics of \eqref{eqn: theorem unnormalized salinity dyanmics} are preserved under the Tschirnhaus transformation because $s$ is simply a translation of $\Delta S$. Moreover, the factor $\beta\lambda^2$ that scales the leading cubic in \eqref{eqn: depressed cubic salinity dynamics} is positive, effectively stretching the cubic vertically.

    We aim to locate pairs of the temperature equilibrium difference $\theta$ and the freshwater forcing $P$ that critically transition the circulation dynamics from a bistable regime to a globally stable regime. This means finding the roots of \eqref{eqn: depressed cubic salinity dynamics} for different values of $P$ and $\theta$.
    Referring to the canonical form \eqref{eqn: cusp canonical form}, this is typically achieved by fixing a value of $r$ and inspecting graphically changes in the intersections of the cubic $g_2 = rw-w^3$ with the horizontal line $g_1 = h$. On the one hand, $r\leq 0$ corresponds to one intersection, i.e. one fixed point; on the other hand $r >0$ corresponds to two or three intersections, where two intersections marks the pitchfork bifurcation value. However, since the coefficients $r$ and $h$ in \eqref{eqn: depressed cubic salinity dynamics} are nonlinear in $\theta$, a given intersection of $g_1$ and $g_2$ may occur for several values of $\theta$. Thus, we instead characterize roots of \eqref{eqn: depressed cubic salinity dynamics} using its cubic discriminant as a function of $P$ and $\theta$.

    The cubic discriminant is given by 
    \begin{equation} \label{eqn: cubic discriminant}
    \Delta (\theta, P) = 4 (r(\theta)/\beta\lambda^2)^3 - 27 (h(\theta, P)/\beta\lambda^2)^2 
    \end{equation}
    Figure \ref{fig: discriminant} demonstrates the discriminant surface of the salininty dynamics as $\theta$ and $P$ are jointly varied. Parameter pairs outside of the region enclosed by the black curve represent parameters which admit a single real root of the right hand side of \eqref{eqn: depressed cubic salinity dynamics}, and pairs inside the region represent parameters which admit three real roots. Pairs lying on the black curve represent parameter pairs that admit two real roots, one of which is of multiplicity two. This leads to  following lemma.

    \begin{lemma}[Discriminant of Salinity Dynamics]\label{lemma: discriminant}
    The solution set of $(\theta, P)$ satisfying the discriminant condition \eqref{eqn: cubic discriminant} is nonempty. For this nontrivial set, the reduced salinity difference dynamics \eqref{eqn: depressed cubic salinity dynamics} varies smoothly from negative to positive.
    \end{lemma}

\begin{figure}[h!]
    \centering
    \includegraphics[width=0.5\linewidth]{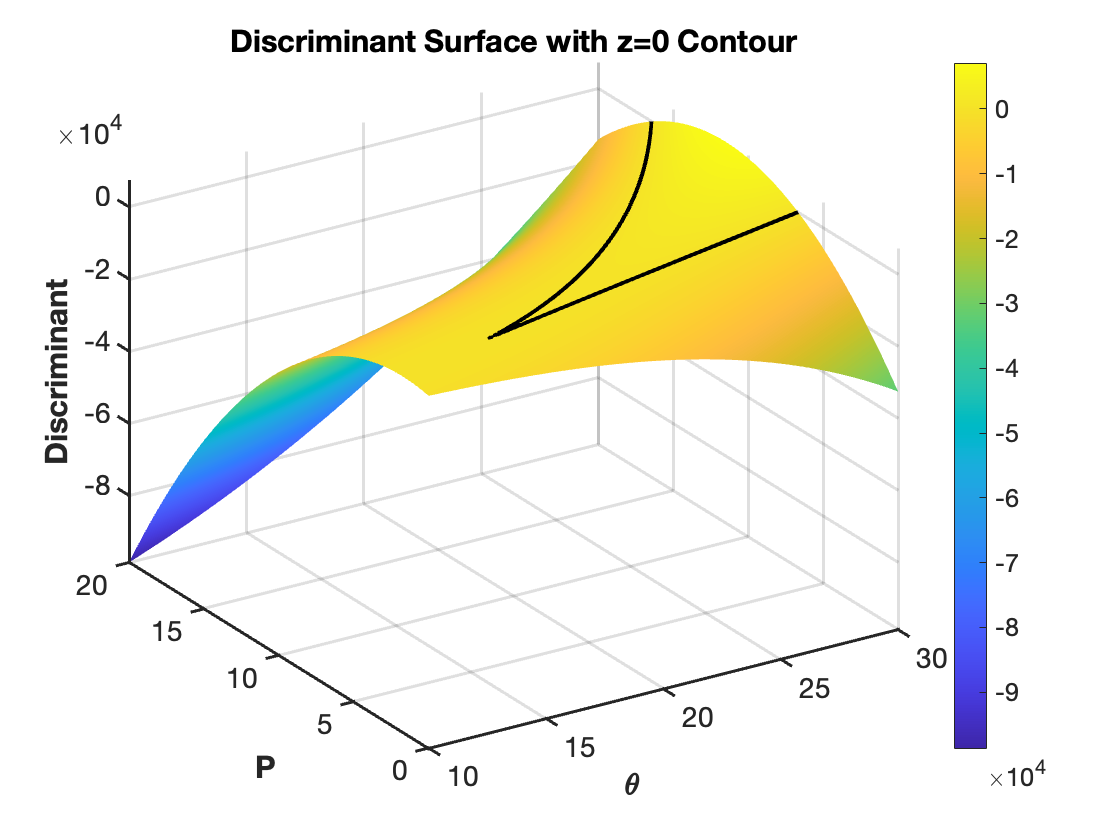}
    \caption{The discriminant surface $\Delta(\theta, P)$ in parameter space. The zero-level contour (black curve) defines the fold curves that bound the region of bistability.}
    \label{fig: discriminant}
\end{figure}

Figure \ref{fig: discriminant} demonstrates a smooth variation of the discriminant $\Delta(\theta, P)$. Variations in the parameters alter the discriminant of the cubic, governing both the number and stability of equilibria and inducing transitions between globally stable and bistable regimes. Crucially, this bifurcation structure is organized by a codimension-two singularity—the cusp point—which serves as the organizing center for the system's dynamics. Thus, the transformed model exhibits a full cusp bifurcation.

\begin{theorem}[Cusp Theorem]\label{thm: cusp theorem}
  The system \cref{eqn: depressed cubic salinity dynamics} undergoes a cusp bifurcation at critical parameters. 
\end{theorem}

\begin{figure}[h!]
    \centering
    \includegraphics[width=0.8\linewidth]{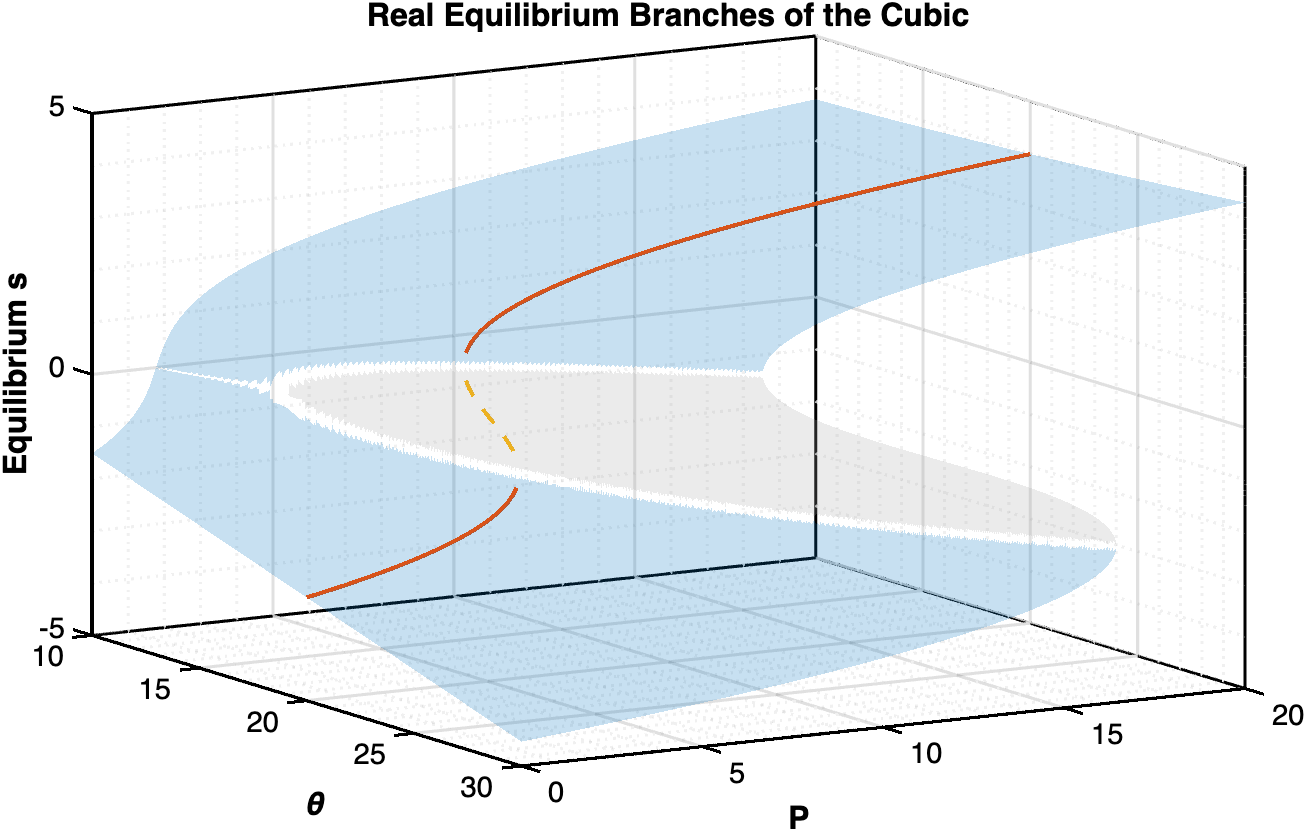}
    \caption{The cusp bifurcation surface for the shifted system \cref{eqn: depressed cubic salinity dynamics}. The equilibrium manifold is organized by a codimension-two cusp singularity, where the pleated region represents bistability between thermally-dominated ($s<0$) and salinity-dominated ($s>0$) modes.}
    \label{fig: cusp}
\end{figure}

The main result is Theorem \ref{thm: cusp theorem}, which establishes the existence of a cusp bifurcation in the one-dimensional Cessi box model \eqref{eqn: depressed cubic salinity dynamics} representing a shifted salinity difference $s$. Under joint variation of freshwater forcing $P$ and the temperature-gradient equilibrium parameter $\theta$, the resulting equilibrium surface, shown in Figure \ref{fig: cusp}, characterizes the dependence of salinity equilibria on these two parameters. Parameter pairs $(\theta, P)$ lying within the folded (pleat) region correspond to bistable regimes, while regions outside the fold admit a single equilibrium. The blue region denotes globally stable equilibria of $s$, and the light grey region denotes unstable equilibria. The two modes of the thermohaline circulation are recovered

The curve corresponding to Cessi’s original bifurcation analysis is recovered as a one-parameter slice of this surface at $\theta=20$ and $P=4.98$. The fold curves delineate critical parameter values at which saddle–node bifurcations occur, with respect to each individual parameter, and their intersection defines the cusp point. Across this boundary, the system transitions from a regime in which equilibria vary smoothly with parameters to one in which small parameter changes produce abrupt transitions between distinct equilibria.

Variations in $\theta$ alter the extent of the bistable region in $P$. For fixed $P=4.98$, increasing $\theta$ shifts the system toward a globally stable equilibrium, while decreasing $\theta$ narrows the range of freshwater forcing admitting bistability. Below a critical value of $\theta$, the fold curves vanish and bistability is lost entirely. These results describe how the location and structure of the cusp depend on both control parameters.


\subsection*{Discussion}
Motivated by observed increases in global mean temperature and polar amplification, we revisited Cessi’s (1994) conceptual model to examine whether changes in the background temperature gradient alone can induce tipping behavior in ocean circulation. Cessi’s original analysis, based on a reduced form of Stommel’s two-box model, exploited a separation of timescales in which the temperature gradient relaxes rapidly to equilibrium relative to salinity. This assumption allowed the system to be reduced to a single prognostic variable—the deviation of salinity from equilibrium—while treating freshwater input to the polar box as a control parameter. Within this framework, Cessi demonstrated that sufficiently strong freshwater forcing could drive the system across a bifurcation, producing abrupt transitions in circulation strength.

In contrast, the present work fixes the freshwater forcing and instead treats the equilibrium temperature gradient as a control parameter. Despite this reversal in perspective, the qualitative behavior of the system remains similar: small, smooth changes in the control parameter can trigger large and abrupt changes in the system’s dynamical state. However, the nature of the response differs depending on the forcing pathway. Freshwater forcing tends to destabilize strong salinity gradients, while temperature-gradient forcing favors transitions toward a single globally stable state characterized by weaker salinity differences. This asymmetry highlights that distinct physical mechanisms can produce qualitatively similar tipping behavior while leading to different equilibrium outcomes.

The cusp bifurcation identified here provides a unifying framework for these behaviors by embedding both forcing mechanisms within a two-parameter stability landscape. Cessi’s original bifurcation diagram emerges as a one-parameter slice of this surface, while the full cusp geometry clarifies how freshwater forcing and background temperature structure interact to govern the existence and extent of bistability. In particular, weakening the temperature gradient contracts the bistable regime and ultimately eliminates it, indicating that the system’s susceptibility to abrupt transitions depends not only on external perturbations but also on the ambient climatic state.

These results suggest that AMOC-like circulation dynamics are vulnerable to tipping through multiple pathways. Changes in background temperature gradients—such as those associated with polar amplification—can either enhance sensitivity to freshwater forcing or remove bistability altogether, fundamentally altering the system’s response to perturbations. More broadly, this work emphasizes that tipping behavior in climate systems should be understood as a multi-parameter phenomenon, where slow, large-scale environmental changes can reshape stability landscapes and modulate the likelihood of abrupt transitions.

\section*{Acknowledgments}
Thank you to Dick McGehee, C. Egan, UMN Math Climate Seminar, and the Mathematics of Climate Research Network for helpful discussions and drafting feedback. This author was partially supported by a GAANN Fellowship under Award P200A240046.

\bibliographystyle{siamplain}
\bibliography{references}

\section*{Supplementary Material}
\input{ex_supplement}

\end{document}

%% file: ex_shared.tex
\usepackage{lipsum}
\usepackage{amsfonts,amsmath, amssymb}
\usepackage{graphicx}
\usepackage{epstopdf}
\usepackage{algorithmic}
\ifpdf
  \DeclareGraphicsExtensions{.eps,.pdf,.png,.jpg}
\else
  \DeclareGraphicsExtensions{.eps}
\fi

\newsiamremark{remark}{Remark}
\newsiamremark{hypothesis}{Hypothesis}
\crefname{hypothesis}{Hypothesis}{Hypotheses}
\newsiamthm{claim}{Claim}
\newsiamremark{fact}{Fact}
\crefname{fact}{Fact}{Facts}
\newcommand{\newreptheorem}[2]{%
\newenvironment{rep#1}[1]{%
 \def\rep@title{#2 \ref{##1}}%
 \begin{rep@theorem}}%
 {\end{rep@theorem}}}
\makeatother   

\newreptheorem{theorem}{Theorem}
\newreptheorem{lemma}{Lemma}

\headers{Cusp Bifurcation in Conceptual Thermohaline Circulation Model}{J. Noory}

\title{Cusp Bifurcation in Conceptual Thermohaline Circulation Model\thanks{Submitted to the editors 02/20206.
\funding{This work was funded in part by US Department of Education GAANN Award P200A240046}}}

\author{Jasmine Noory\thanks{University of Minnesota 
  (\email{noory003@umn.edu}\url{}).}}

\usepackage{amsopn}


%% file: ex_supplement.tex

\section{Model Reformulation}

\begin{lemma} \label{lemma: new stommel formulation}
    For a large enough timescale distinction, a one-dimensional model of thermohaline circulation, (2.4) 
    may be restated as
    \begin{equation}
        \frac{d\Delta S}{dt'} =P - \Delta S(1+\beta(\theta -\lambda \Delta S)^2 ) \label{eqn: theorem unnormalized salinity dyanmics}
    \end{equation}
    with the parameters $\lambda = \frac{\alpha_S}{\alpha_T}$, $\beta = \frac{qt_d \alpha_T^2}{V}$, and $P = \frac{S_0t_d}{H}\bar F$, where $\bar F$ is an annual average freshwater flux, defined by Cessi.
\end{lemma}
\vspace{.5cm}

\begin{proof}
Begin with nondimensional system of equations of stommel's box model from Cessi's (2.3), 

\begin{align}
    \frac{dx}{dt'} &= -\alpha( x-1) - x(1+\mu^2(x-y)^2 ) \label{eqn:cessi temp}\\ 
    \frac{dy}{dt'} &= p - y (1+\mu^2(x-y)^2) \label{eqn:cessi salinity}
\end{align}

where $\alpha, \mu^2,$ and $p(t)$ are all nondimensional parameters described by 
\[\alpha = \frac{t_d}{t_r} \ \ \ \ \mu^2 = \frac{  qt_d (\alpha_T\theta)^2}{V}, \ \ \ \ p (t) = \frac{\alpha_S S_0 t_d }{\alpha_T\theta H }F(t).\]

This formulation is obtained by taking the following change of variables from Stommel's original model (add reference here)
\begin{align}
    x=\frac{\Delta T}{\theta} \ \ \ y=\frac{\alpha_S\Delta S}{\alpha_T \theta} \ \ \ t= t_d t' \label{eqn: cessi change of variables}
\end{align}

We pull back to the unnormalized temperature difference and salinity difference using the the relations \ref{eqn: cessi change of variables}. This gives that 
\[\frac{d \Delta T}{dt'} = \theta \frac{dx}{dt'} \hspace{1cm} \frac{d\Delta S}{dt'} = \frac{\alpha_T \theta}{\alpha_S}\frac{dy}{dt'}. \]

Explicitly, we have a system of two nonlinear equations in dimensional form with $\frac{d \Delta T}{dt'}$ describing temperature difference change in celcius and $\frac{d \Delta S}{dt'} $ the salinity difference change in practical salinity unit (psu). This model is described with the following equations.

\begin{align}
     \frac{d\Delta T}{dt'} &= -\alpha(\Delta T - \theta) - \Delta T(1+\beta(\Delta T- \lambda \Delta S)^2 ) \label{eqn:temperature difference}\\
     \frac{d\Delta S}{dt'} &=P - \Delta S(1+\beta(\Delta T -\lambda \Delta S)^2 )\label{eqn:salinity difference}
\end{align}

With the parameters $\lambda = \frac{\alpha_S}{\alpha_T}$, $\beta = \frac{qt_d \alpha_T^2}{V}$, and $P = \frac{S_0t_d}{H}F$. Notably, we've recovered the temperature difference equilibrium $\theta$, which represents the difference of ambient temperatures surrounding the northern and southern boxes to which temperatures relax. Moreover, the fast-slow properties of the dynamics are preserved since $\alpha$ is left unchanged.

The system may still be reduced to a one dimensional equation, since $\alpha$ is very large. Because the temperature difference $\Delta T$ evolves at a faster timescale relative to salinity difference $\Delta S$, we can further simplify the analysis of the system by projecting the salinity dynamics onto the critical manifold of the temperature difference, where the fast variable is at quasi-equilibrium $\Delta T \approx \theta$. The one dimensional system is as follows.
\begin{align}
    \frac{d\Delta S}{dt'} &=P - \Delta S(1+\beta(\theta -\lambda \Delta S)^2 )
\end{align}

Thus we have a one dimensional model of the salinity dynamics where the temperature difference equilibrium $\theta$ is a parameter. In this view, the parameter $\lambda$ describes the ratio of haline contration to thermal expansion. The density difference term $(\Delta T - \lambda\Delta S)^2$, now described in temperature units, still drives the dimensionless exchange $(1 + \beta(\Delta T - \lambda\Delta S)^2 )$. 
\end{proof}

\vspace{1cm}

\section{Temperature Induced Saddle-Node Bifurcation}\label{sec:temperaturesaddle}\begin{theorem}\label{thm: temperaturesaddle}The equilibrium of the reformulated model \cref{eqn: theorem unnormalized salinity dyanmics} undergoes saddle-node bifurcations with respect to the temperature parameter $\theta$, creating a region of drop from bistability.\end{theorem}
\vspace{.5cm}
\begin{proof}
Let $S \equiv \Delta S$ and let $S^*$ be an equilibrium satisfying $f(S^*, \theta, P) = 0$, where the dynamics are given by:$$ f(S, \theta, P) = P - S \left[1 + \beta (\theta - \lambda S)^2 \right]. $$ 
A saddle-node bifurcation occurs when the linear stability of the equilibrium degenerates, i.e., $\frac{\partial f}{\partial S} \big|_{(S^*, \theta_c)} = 0$. Computing the partial derivative:
\begin{align*}
\frac{\partial f}{\partial S} &= -1 - \beta (\theta - \lambda S)^2 + 2\beta  \lambda S (\theta - \lambda S) \\[9pt]
&= - \beta \theta^2 + 4\beta \lambda S \theta - (1+ 3\beta \lambda^2 S^2).\end{align*}

Setting this to zero yields a quadratic equation in $\theta_c$:

\begin{equation} \label{eqn:critical_condition}
- \beta \theta^2 + 4\beta \lambda S \theta - (1+ 3\beta \lambda^2 S^2) = 0.
\end{equation}

Solving for $\theta_c$, we obtain:$$ \theta_c = \frac{-4\beta \lambda S^* \pm \sqrt{16\beta^2 \lambda^2 S^{*2} - 4\beta(3\beta \lambda^2 S^{*2} + 1)}}{-2\beta}. $$Simplifying the discriminant:
\begin{equation}
\theta_c = 2\lambda S^{*} \mp \sqrt{(\lambda S^*)^2 - \frac{1}{\beta}}.
\end{equation}
The existence of two real roots for $\theta_c$ (provided $\lambda^2 S^{*2} > 1/\beta$) indicates two turning points in the bifurcation diagram, which bound the hysteresis loop.
\end{proof}
\vspace{1cm}

\vspace{1cm}
\section{Tschirnhaus Transformation}
\begin{lemma}
    The Tschirnhaus transformation of equation \cref{eqn: theorem unnormalized salinity dyanmics} is \begin{equation*}
            s=\Delta S+\frac{2\theta}{3\lambda}.
        \end{equation*}
\end{lemma}

\vspace{.5cm}
\begin{proof}
    Let $\Delta S = (s + k)$, and observe that \cref{eqn: theorem unnormalized salinity dyanmics} becomes 

\begin{align}
    \dot s &= P(t) - (s + k)\Bigl(1+\beta\bigl(\theta -\lambda (s + k) \bigr)^2 \Bigr) \notag \\[9pt]
           &= P(t) - (1+\beta \theta^2)(s+k) + 2\beta \lambda \theta (s+k)^2 - \beta \lambda ^2(s+k)^3 \notag \\[9pt]
           &= P(t) - (1+\beta \theta^2)(s+k) + (2 \beta \lambda \theta k^2 + 4 \beta \lambda \theta k s + 2 \beta \lambda \theta s^2) \notag \\
           &\quad - (\beta \lambda^2 k^3 + 3 \beta \lambda^2 k^2 s + 3 \beta \lambda^2 k s^2 + \beta \lambda^2 s^3) \notag
\end{align}
A depressed cubic form requires that $s^2$ term should be suppressed. We will collect terms of $s^2$ and solve for $k$ which enforces a zero coefficient. That is to say we require that 

\[-3\beta \lambda^2 k + 2\beta \lambda \theta = 0 \Rightarrow k = \frac{2\theta}{3\lambda} \]

Therefore, $\Delta S = s + \frac{2\theta}{3\lambda}$ is the Tschirnhaus transformation which depresses the cubic. 
\end{proof}

\vspace{1cm}
\section{Depressed Dynamics}
\begin{corollary}
    The salinity difference dynamics \eqref{eqn: theorem unnormalized salinity dyanmics} expressed in the shifted variable $s$ is 
        \begin{equation}\label{eqn: depressed cubic salinity dynamics}
            \dot s = h(\theta, P) + r(\theta) s  - \beta \lambda^2s^3 ,
        \end{equation}
        where $r(\theta) = (\frac{\beta\theta^2}{3}-1)  $ and $h(\theta, P) = P - \frac{2\theta}{3 \lambda}-\frac{2 \beta \theta^3}{27 \lambda}$.
\end{corollary}

\vspace{.5cm}
\begin{proof}
    Let $\Delta S = s + \frac{2\theta}{3\lambda}$. Then, \cref{eqn: theorem unnormalized salinity dyanmics} becomes

    \begin{align}
    \dot s &= P(t) - (s + \frac{2\theta}{3\lambda})(1+\beta(\theta -\lambda (s + \frac{2\theta}{3\lambda} ) )^2 ) \notag \\[9pt]
    \dot s &= P - \frac{2\theta}{3 \lambda}-\frac{2 \beta \theta^3}{27 \lambda} + s(\frac{\beta\theta^2}{3}-1) - \beta \lambda^2s^3 \notag \\[9pt]
    \dot s &= P - \frac{2\theta}{3 \lambda}-\frac{2 \beta \theta^3}{27 \lambda} + s [ (\frac{\beta\theta^2}{3}-1) - \beta \lambda^2s^2] \label{eqn: depressed cubic salinity dynamics}
\end{align}
\end{proof}

\vspace{1cm}
\section{Discriminant of the Salinity Dyanmics}
    \begin{lemma}
        The solution set of $(\theta, P)$ satisfying the discriminant condition    
        \begin{align}
        \Delta (\theta, P) = 4 (r(\theta)/\beta\lambda^2)^3 - 27 (h(\theta, P)/\beta\lambda^2)^2
        \end{align} is nonempty. For this nontrivial set, the reduced salinity difference dynamics \eqref{eqn: depressed cubic salinity dynamics} varies smoothly from negative to positive.
    \end{lemma}

\vspace{.5cm}
\begin{proof}
    The bifurcation set is the collection of pairs $(\theta, P)$ such that the cubic discriminant of the vectorfield, $\Delta(\theta, P),$ is $0$. We will find the roots of our salinity difference vectorfield, i.e. $h(\theta, P) + r(\theta)s - \beta\lambda^2s^3 = 0$ and inspect where the discriminant transitions from negative, i.e. where there is only one real fixed point, to positive, where there are three fixed points. 
    
    Recall that the discriminant of a depressed cubic polynomial $y^3 + Ay+B=0$ is $\Delta = -4A^3-27B^2$.

    Taking $A = - \frac{r(\theta)}{\beta\lambda^2}$ and $B =- \frac{h(\theta,P)}{\beta\lambda^2}$,  the discriminant is
    \begin{align}
        \Delta (\theta, P) = 4 (r(\theta)/\beta\lambda^2)^3 - 27 (h(\theta, P)/\beta\lambda^2)^2
    \end{align}
    
    \begin{figure}[h!]
        \centering
        \includegraphics[width=0.5\linewidth]{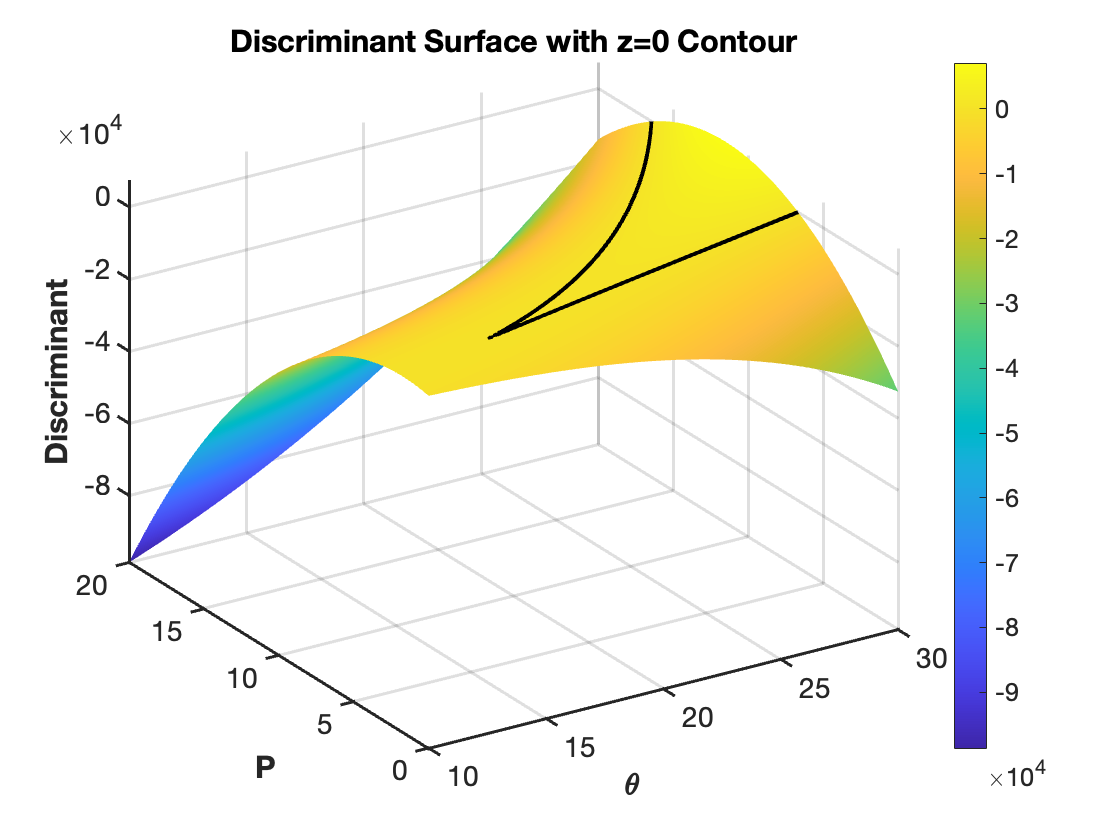}
        \caption{The discriminant surface generated by parameters $P$ and $\theta$. }
        \label{fig: discriminant}
    \end{figure}
    
    The discriminant smoothly varies from negative to positive. We see that for there exist pairs of $(\theta,P)$ such that the discriminant is zero, see figure \ref{fig: discriminant}. This points to the existence of a pitchfork bifurcation.  
\end{proof}

\vspace{1cm}

\section{Cusp Theorem}
\begin{theorem}\label{thm: cusp theorem}
 Let $S \equiv \Delta S$. The system governed by the dynamics $f(S, \theta, P) = P - S[1+\beta(\theta-\lambda S)^2]$ exhibits a cusp bifurcation at the unique parameter point:
\begin{equation}
(\theta_{\text{cusp}}, P_{\text{cusp}}) = \left( \sqrt{\frac{3}{\beta}}, \quad \frac{8}{3\lambda\sqrt{3\beta}} \right).
\end{equation}
This point marks the boundary in parameter space separating the regime of hysteresis from the regime of monotonic stability.. 
\end{theorem}

\begin{proof}
    We seek the codimension-2 point where the hysteresis loop collapses. This occurs when the two saddle-node bifurcations derived in \cref{thm: temperaturesaddle} merge into a single point. Analytically, this requires a higher-order degeneracy where not only the linear stability vanishes, but the curvature of the vector field vanishes as well.
Let $S \equiv \Delta S$. The dynamics are given by $f(S, \theta, P) = P - S(1 + \beta (\theta - \lambda S)^2)$. The cusp geometry requires the simultaneous satisfaction of three conditions:
\begin{enumerate}
    \item Equilibrium: $f = 0$
    \item Vanishing Linear Stability: $\frac{\partial f}{\partial S} = 0$
    \item Vanishing Curvature: $\frac{\partial^2 f}{\partial S^2} = 0$
\end{enumerate}

\noindent We first examine the curvature condition. Computing the second derivative of $f$ with respect to $S$:
\[ \frac{\partial^2 f}{\partial S^2} = 2\beta \lambda \bigl( 2(\theta - \lambda S) - \lambda S \bigr). \]
Setting this to zero yields a geometric constraint on the state space:
\begin{equation}\label{eqn:cusp_constraint}
    S = \frac{2\theta}{3\lambda}.
\end{equation}

\noindent Next, we substitute \cref{eqn:cusp_constraint} into the linear stability condition, $\frac{\partial f}{\partial S} = 0$. Using the substitution $u = \theta - \lambda S = \frac{\lambda S}{2}$, the condition simplifies:
\begin{align}
    \frac{\partial f}{\partial S} &= -1 - \beta u^2 + 2\beta \lambda S u \nonumber \\
                                           &= -1 - \beta \left(\frac{\lambda S}{2}\right)^2 + \beta (\lambda S)^2 \nonumber \\
                                           &= -1 + \frac{3}{4}\beta \lambda^2 S^2 = 0.
\end{align}
This fixes the critical salinity difference at the cusp:
\begin{equation}
    S_c = \frac{2}{\lambda \sqrt{3\beta}}.
\end{equation}

\noindent Finally, we recover the critical parameters. Substituting $S_c$ back into \cref{eqn:cusp_constraint} gives the critical temperature forcing:
\[ \theta_c = \frac{3\lambda}{2} S_c = \sqrt{\frac{3}{\beta}}. \]
Substituting both $S_c$ and $\theta_c$ into the equilibrium condition $f=0$ determines the critical freshwater flux $P$:
\[ P_c = S_c \left( 1 + \beta \left( \frac{\lambda S_c}{2} \right)^2 \right) = \frac{4}{3} S_c = \frac{8}{3\lambda\sqrt{3\beta}}. \]

\noindent Since the third derivative $\frac{\partial^3 f}{\partial S^3} = -6\beta \lambda^2 \neq 0$, the non-degeneracy condition is met, confirming this point is a cusp bifurcation.
\end{proof}

%% file: references.bib
@article{Cessi1994,
  author = {Cessi, Paola},
  title = {A Simple Box Model of Stochastically Forced Thermohaline Circulation},
  journal = {Journal of Physical Oceanography},
  volume = {24},
  number = {9},
  pages = {1911--1920},
  year = {1994}
}

@article{Stommel1961,
  author = {Stommel, Henry},
  title = {Thermohaline convection with two stable regimes of flow},
  journal = {Tellus},
  volume = {13},
  number = {2},
  pages = {224--230},
  year = {1961}
}

@book{Strogatz,
isbn = {0367026503},
language = {eng},
publisher = {CRC Press},
title = {Nonlinear dynamics and chaos: with applications to physics, biology, chemistry, and engineering},
year = {2024},
author = {Strogatz, Steven Henry},
address = {Boca Raton},
edition = {Third edition.},
}

@misc{Zeeman1976,
issn = {0036-8733},
journal = {Scientific American},
language = {eng},
number = {4},
pages = {65-83},
title = {Catastrophe Theory},
volume = {234},
year = {1976},
author = {Zeeman, E. C.},
}

@book{ThomRené1989,
isbn = {9780201094190},
language = {eng},
publisher = {Addison-Wesley, Advanced Book Program},
series = {Advanced book classics},
title = {Structural stability and morphogenesis: an outline of a general theory of models},
year = {1989},
author = {Thom, René},
address = {Redwood City, Calif},
}

@misc{NOAA,
  author = {{National Oceanic and Atmospheric Administration}},
  title = {What is the Atlantic Meridional Overturning Circulation?},
  year = {2023},
  howpublished = {\url{https://oceanservice.noaa.gov/facts/amoc.html}},
  note = {Accessed: 2025-05-15}
}

@article{Fedorov,
author = {Sévellec, Florian and Fedorov, Alexey V.},
address = {Boston, MA},
copyright = {2013 American Meteorological Society},
issn = {0894-8755},
journal = {Journal of climate},
language = {eng},
number = {7},
pages = {2160-2183},
publisher = {American Meteorological Society},
title = {The Leading, Interdecadal Eigenmode of the Atlantic Meridional Overturning Circulation in a Realistic Ocean Model},
volume = {26},
year = {2013},
}

@article{Serreze,
author = {Serreze, Mark C. and Barry, Roger G.},
address = {AMSTERDAM},
copyright = {2011},
issn = {0921-8181},
journal = {Global and planetary change},
language = {eng},
number = {1},
pages = {85-96},
publisher = {Elsevier B.V},
title = {Processes and impacts of Arctic amplification: A research synthesis},
volume = {77},
year = {2011},
}

@article{Rahmstorf1996,
author = {Rahmstorf, S.},
address = {NEW YORK},
copyright = {Copyright 2017 Elsevier B.V., All rights reserved.},
issn = {0930-7575},
journal = {Climate dynamics},
language = {eng},
number = {12},
pages = {799-811},
publisher = {Springer Nature},
title = {On the freshwater forcing and transport of the Atlantic thermohaline circulation},
volume = {12},
year = {1996},
}
